\documentclass[12pt]{amsart}
\usepackage{a4wide}
\usepackage[T1]{fontenc}
\usepackage{amssymb,amsmath,amsthm,latexsym}
\usepackage{mathrsfs}
\usepackage[usenames,dvipsnames]{color}
\usepackage{euscript}
\usepackage{graphicx}
\usepackage{mdwlist}
\usepackage{mathtools,dsfont,wasysym}
\usepackage{stmaryrd}
\usepackage[dvipsnames]{xcolor}

\DeclareFontFamily{U}{mathx}{\hyphenchar\font45}
\DeclareFontShape{U}{mathx}{m}{n}{
      <5> <6> <7> <8> <9> <10>
      <10.95> <12> <14.4> <17.28> <20.74> <24.88>
      mathx10
      }{}

\usepackage{todonotes}

\newtheorem{theorem}{Theorem}[section]
\newtheorem*{theoremA}{Theorem A}
\newtheorem*{theoremB}{Theorem B}

\newtheorem{lemma}[theorem]{Lemma}
\newtheorem{corollary}[theorem]{Corollary}
\newtheorem{proposition}[theorem]{Proposition}
\newtheorem{fact}[theorem]{Fact}
\theoremstyle{remark}
\newtheorem{remark}[theorem]{Remark}
\newtheorem*{claim*}{Claim}
\theoremstyle{definition}
\newtheorem{definition}[theorem]{Definition}
\newtheorem{problem}[theorem]{Problem}

\numberwithin{equation}{section}
\makeatother

\newcommand{\nn}[1]{{\left\vert\kern-0.25ex\left\vert\kern-0.25ex\left\vert #1 
\right\vert\kern-0.25ex\right\vert\kern-0.25ex\right\vert}}
\newcommand{\embed}{\hookrightarrow}
\renewcommand{\leq}{\leqslant}
\renewcommand{\geq}{\geqslant}

\newcounter{smallromans}

\newenvironment{romanenumerate}
{\begin{list}{{\normalfont\textrm{(\roman{smallromans})}}}%
  {\usecounter{smallromans}\setlength{\itemindent}{0cm}%
   \setlength{\leftmargin}{5.5ex}\setlength{\labelwidth}{5.5ex}%
   \setlength{\topsep}{.5ex}\setlength{\partopsep}{.5ex}%
   \setlength{\itemsep}{0.1ex}}}%
{\end{list}}

\newcommand{\R}{\mathbb{R}}
\newcommand{\N}{\mathbb{N}}

\newcommand{\n}{\left\Vert\cdot\right\Vert}
\newcommand{\F}{\mathcal{F}}
\newcommand{\Fc}{\overline{\mathcal{F}}}

\newcounter{smallromansdash}

{\end{list}}

\newcounter{bigromans} 
{\end{list}}

\begin{document}
\title[An uncountable version of Pt\'ak's lemma]{An uncountable version of\\Pt\'ak's combinatorial lemma}

\author[P.~H\'ajek]{Petr H\'ajek}
\address[P.~H\'ajek]{Mathematical Institute\\Czech Academy of Science\\\v Zitn\'a 25 \\115 67 Praha 1\\
Czech Republic and Department of Mathematics\\Faculty of Electrical Engineering\\
Czech Technical University in Prague\\Technick\'a 2, 166 27 Praha 6\\ Czech Republic}
\email{hajek@math.cas.cz}

\author[T.~Russo]{Tommaso Russo}
\address[T.~Russo]{Department of Mathematics\\Faculty of Electrical Engineering\\
Czech Technical University in Prague\\Technick\'a 2, 166 27 Praha 6\\ Czech Republic and Dipartimento di matematica\\ Universit\`a degli Studi di Milano\\
via Saldini 50, 20133 Milano, Italy}
\email{tommaso.russo@unimi.it}

\thanks{Research of the first-named author was supported in part by OPVVV CAAS CZ.02.1.01/0.0/0.0/16$\_$019/0000778 and by GA\v CR 16-07378S, RVO: 67985840.
Research of the second-named author was supported by the project International Mobility of Researchers in CTU CZ.02.2.69/0.0/0.0/16$\_$027/0008465 and by Gruppo Nazionale per l'Analisi Matematica, la Probabilit\`a e le loro Applicazioni (GNAMPA) of Istituto Nazionale di Alta Matematica (INdAM), Italy.}

\keywords{Pt\'ak's combinatorial lemma, convex mean, adequate family, Corson compact, Erd\H{o}s space, pseudo-weight}
\subjclass[2010]{05A20, 03E35 (primary), and 46B26, 46E15 (secondary).}
\date{\today}

\begin{abstract} In this note we are concerned with the validity of an uncountable analogue of a combinatorial lemma due to Vlastimil Pt\'ak. We show that the validity of the result for $\omega_1$ can not be decided in ZFC alone. We also provide a sufficient condition, for a class of larger cardinals.
\end{abstract}
\maketitle

\section{Introduction}
In his 1959 paper \cite{Ptak lemma}, Vlastimil Pt\'ak distilled a combinatorial lemma aimed at the investigation of weak compactness in Banach spaces. An interesting application of the lemma, and partial motivation for the result itself, was an elementary proof of the fact that if a uniformly bounded sequence of continuous functions $(f_n)_{n=1}^\infty\subseteq C(K)$ converges pointwise to a continuous function $f$, then $f$ may be uniformly approximated by convex combinations of the $f_n$ (\cite[\S 2.1]{Ptak w-cpt}). It is actually a standard exercise in Functional Analysis to understand this assertion as a particular case of Mazur's theorem that a closed and convex subset of a Banach space is weakly closed. However, this approach requires the Riesz representation theorem for $C(K)^*$, Lebesgue's dominated convergence theorem and the Hahn--Banach separation theorem; therefore, it relies on much deeper principles than the assertion itself, which, in particular, involves no measure theory whatsoever.

In later papers, Pt\'ak also applied and extended the same combinatorial ideas to include a treatment of weak compactness in terms of inverting the order of two limit processes, \emph{cf.} \cite{Ptak extension1, Ptak extension2}. This approach, already present in \cite{Ptak w-cpt}, is also followed in the monograph \cite[\S 24.6]{Kothe}, for the proof of Krein's theorem.\smallskip

This interesting lemma attracted the attention of the mathematical community, as witnessed by several papers dedicated to its different proofs or extensions; let us mention, among them, \cite{Kindler, Simons lattice, Simons related, Simons ptak}. More recently, it was also used---and given a different, Banach space theoretic, proof---in the paper \cite{BHO newPtak}. This proof is also included in \cite[Exercise 14.28]{FHHMZ}, where Pt\'ak's elementary proof of Mazur theorem is also outlined (\emph{cf.} \cite[Exercise 14.29]{FHHMZ}). One further introduction to this result may be found in \cite[\S I.3]{Todorcevic ptak}, or in the systematic survey \cite{Ptak survey} by Pt\'ak himself. \smallskip

Let us now proceed to recall the statement of the result under investigation; we shall require a piece of terminology, and we follow Pt\'ak's notation from \cite{Ptak lemma}. Given a set $S$ and a function $\lambda\colon S\to\R$ by the \emph{support} of $\lambda$ we understand the set ${\rm supp}(\lambda):=\{s\in S\colon \lambda(s)\neq0\}$; in the case that ${\rm supp}(\lambda)$ is a finite set, we shall say that $\lambda$ is \emph{finitely supported}.

A \emph{convex mean} is a finitely supported function $\lambda\colon S\to[0,\infty)$ such that $$\sum_{s\in S}\lambda(s)=1.$$
Plainly, a convex mean can also be naturally interpreted as a finitely supported probability measure on $\left(S,2^S\right)$ via the definition $\lambda(A):=\sum_{s\in A}\lambda(s)$, for $A\subseteq S$. In what follows, we shall profit from this notation, whenever convenient.

Given a set $S$, we shall also denote by $[S]^{<\omega}$ the collection comprising all finite subsets of $S$. All the necessary notation being set forth, we are now in position to recall the original statement of Pt\'ak's lemma.
\begin{lemma}[Pt\'ak's combinatorial lemma, \cite{Ptak lemma}]\label{Ptak I} Let $S$ be an infinite set and $\F\subseteq[S]^{<\omega}$ be a collection of finite subsets of $S$. Then the following conditions are equivalent:
\begin{romanenumerate}
\item there exist an infinite subset $H$ of $S$ and $\delta>0$ such that for every convex mean $\lambda$ with ${\rm supp}(\lambda)\subseteq H$ one has $$\sup_{F\in\F}\lambda(F)\geq\delta;$$
\item there exist a strictly increasing sequence of finite sets $(B_n)_{n=1}^\infty\subseteq[S]^{<\omega}$ and a sequence $(F_n)_{n=1}^\infty\subseteq\F$ such that $B_n\subseteq F_n$, for every $n\in\N$.
\end{romanenumerate}
\end{lemma}

Let us observe that the proof of the implication (ii)$\implies$(i) is immediate, as witnessed by the choice of the set $H:=\cup_{n=1}^\infty B_n$. Therefore, the actual content of the lemma lies in the validity of the implication (i)$\implies$(ii) and it is precisely this implication to appear in the result as devised in \cite{BHO newPtak}.

In order to state this second formulation, we need one more definition. A family $\F\subseteq2^S$ is said to be \emph{hereditary} if whenever $F\in\F$ and $G\subseteq F$, then $G\in\F$ too.
\begin{lemma}[Pt\'ak's lemma, second formulation, \cite{BHO newPtak}]\label{Ptak II} Let $S$ be an infinite set and let $\F\subseteq[S]^{<\omega}$ be an hereditary family. Assume that there exists $\delta>0$ such that for every convex mean $\lambda$ on $S$ one has
$$\sup_{F\in\F}\lambda(F)\geq\delta.$$
Then there exists an infinite subset $M$ of $S$ such that every finite subset of $M$ is in $\F$.
\end{lemma}

Observe that, for an hereditary family $\F$, condition (ii) of Lemma \ref{Ptak I} is equivalent to the conclusion of Lemma \ref{Ptak II}, as a simple verification shows. More precisely, the condition $\F$ being hereditary is only used in the verification that (ii) implies the conclusion of Lemma \ref{Ptak II}. Moreover, the assumption in Lemma \ref{Ptak II} immediately implies (i) with $H=S$ and, conversely, under the validity of (i), the assumption of Lemma \ref{Ptak II} is satisfied for the infinite set $H$ and the hereditary family $\{F\cap H\colon F\in\F\}=\{F\in\F\colon F\subseteq H\}$.

Consequently, the two statements are formally equivalent. The advantage of the second formulation, from our perspective, is that it immediately suggests its possible generalisations to larger cardinalities, that we shall consider in our note.\smallskip

In order to have a more succinct formulation of our results, the following definition seems appropriate.
\begin{definition} Let $\kappa$ be an infinite cardinal number. We say that \emph{Pt\'ak's lemma holds true for} $\kappa$ if for every set $S$ with $|S|\geq\kappa$ and every hereditary family $\F\subseteq[S]^{<\omega}$ such that
$$(\dagger)\qquad\delta:=\inf\left\{ \sup_{F\in\F}\lambda(F)\colon\lambda\,\textrm{ is a convex mean on}\,S\right\}>0,$$
there is a subset $M$ of $S$, with $|M|=\kappa$, such that every finite subset of $M$ belongs to $\F$.
\end{definition}

Observe that in the above definition it would be equivalent to require the condition $|S|=\kappa$, on the set $S$; this is proved by a very simple argument, \emph{cf.} the first part of the proof of Lemma \ref{Ptak II}. Let us now proceed to state our main results.
\begin{theoremA} The validity of Pt\'ak's lemma for $\omega_1$ is independent of ZFC. More precisely:
\begin{romanenumerate}
\item {\rm (MA$_{\omega_1}$)} Pt\'ak's lemma holds true for $\omega_1$;
\item {\rm (CH)} Pt\'ak's lemma fails to hold for $\omega_1$.
\end{romanenumerate}
\end{theoremA}

\begin{theoremB} Let $\kappa$ be a regular cardinal number such that $\lambda^\omega<\kappa$ whenever $\lambda<\kappa$. Then Pt\'ak's lemma is true for $\kappa$.
\end{theoremB}

Note that, if $\mu$ is any infinite cardinal number, then $\kappa:=(2^\mu)^+$ satisfies the assumptions of Theorem B. Consequently, there are in ZFC arbitrarily large cardinal numbers for which Pt\'ak's lemma is true. Moreover, the smallest cardinal Theorem B applies to is $\mathfrak{c}^+$. The next corollary is also a consequence of the theorem.
\begin{corollary}${}$
\begin{romanenumerate}
    \item Pt\'ak's lemma is true for $(2^\mu)^+$, whenever $\mu$ is an infinite cardinal number.
    \item {\rm (CH)} Pt\'ak's lemma is true for $\omega_2$;
    \item {\rm (GCH)} If $\tau$ is a cardinal number with ${\rm cf}(\tau)>\omega$, then Pt\'ak's lemma is true for $\tau^+$.
\end{romanenumerate}
\end{corollary}

The paper is organised as follows: in Section \ref{Sec2 General} we shall present some general observations concerning the condition appearing in Pt\'ak's lemma. These considerations will, in particular, allow us to present the proof of Pt\'ak's original result and are based on the proof given in \cite{BHO newPtak}. In Section \ref{Sec3 Omega1} we shall prove Theorem A, while Section \ref{Sec4 Larger} is dedicated to the proof of Theorem B.

In conclusion to this section, we mention that the notation $\lambda^\omega$ refers to cardinal exponentiation. Moreover, all the topological spaces relevant for this note will be Hausdorff and we will tacitly assume this property throughout the paper. We refer, \emph{e.g.}, to \cite{Jech} for our notation on set theory and to \cite{ak, FHHMZ, HMVZ} for unexplained notation concerning Banach spaces.

\section{General remarks}\label{Sec2 General}
If $S$ is any set, a subset $A$ of $S$ can be naturally identified, via the correspondence $A\mapsto\chi_A$, with an element of the compact topological space $\{0,1\}^S$, endowed with the canonical product topology. Let us recall that, under this identification, if $A\in2^S$, a basis of neighborhoods of $A$ is given by the collection of sets
$$\left\{B\in2^S\colon F\subseteq B\subseteq S\setminus G\right\},$$
where $F$ and $G$ are finite subsets of $A$ and $S\setminus A$ respectively. Throughout our note, we shall make this identification and we shall not distinguish between the set $A$ and its characteristic function $\chi_A$. Therefore, when $\F\subseteq2^S$, we may consider the closure $\Fc$ of $\F$, in the product topology of $\{0,1\}^S$; henceforth, whenever we use the notation $\Fc$ it will be the product topology the one under consideration.\smallskip

In the case that $\F$ is an hereditary family, it is easily seen that $\Fc$ is an \emph{adequate compact}, in the sense of the following definition, first introduced by Talagrand \cite{Talagrand adeq1, Talagrand adeq2}. A family $\mathcal{G}\subseteq2^S$ is said to be \emph{adequate} if:
\begin{romanenumerate}
\item whenever $G\subseteq F$ and $F\in\mathcal{G}$, then $G\in\mathcal{G}$, \emph{i.e.}, $\mathcal{G}$ is hereditary;
\item if every finite subset of $G$ belongs to $\mathcal{G}$, then $G\in\mathcal{G}$ too.
\end{romanenumerate}
Conversely, every adequate family $\mathcal{G}$ can be expressed as $\Fc$, for some hereditary family of finite sets, namely $\F=\{F\in\mathcal{G}\colon|F|<\omega\}$; in particular, every adequate family is a closed subset of $\{0,1\}^S$. As it turns out, compact sets that originate from adequate families of sets are very fascinating objects in Functional Analysis and have been exploited in several important examples; let us refer, \emph{e.g.}, to \cite{Arg Arv Merc, AMN88, Arkhangelskii, benyaministarbird, Leiderman, Leiderman Sokolov, Plebanek, Talagrand adeq1, Talagrand adeq2} for a sample of some of these constructions.\smallskip

Our interest in adequate families originates from the following fact, a particular case of the observation that $\Fc$ is adequate, whenever $\F$ is hereditary.
\begin{fact}\label{Fact: from Fclosure to Ptak} Let $\F$ be an hereditary family and $M\in\Fc$. Then every finite subset of $M$ belongs to $\F$.
\end{fact}

The next proposition is the non-separable counterpart to the argument in \cite[Lemma 3.1]{BHO newPtak}, with the same proof, which we include for the sake of completeness.
\begin{proposition}\label{Prop: Ptak gives l1} Let $S$ be an infinite set and $\F\subseteq[S]^{<\omega}$ be an hereditary family such that $(\dagger)$ holds. Then $C(\Fc)$ contains an isomorphic copy of $\ell_1(S)$.
\end{proposition}
\begin{proof} Let us preliminarily note that if $\lambda$ is any convex mean on $S$, then $\sup_{F\in\F}\lambda(F)\geq\delta$; since this supremum is actually over the finite set consisting of all $F\subseteq{\rm supp}(\lambda)$, it follows that there exists $F\in\F$ with $\lambda(F)\geq\delta$. Consequently, for every finitely supported function $\lambda\colon S\to[0,\infty)$ there exists $F\in\F$ such that
$$\sum_{s\in F}\lambda(s)\geq\delta\cdot\sum_{s\in S}\lambda(s).$$

For an element $x=(x(s))_{s\in S}\in c_{00}(S)$, let us define 
$$\|x\|:=\sup\left\{\left|\sum_{s\in F}x(s)\right| \colon F\in\F\right\};$$
we claim that $\n$ is a norm on $c_{00}(S)$, equivalent to the $\n_1$ norm. In order to prove this, fix $x\in c_{00}(S)$ and let $P$ be the finite set $P:=\{s\in S\colon x(s)>0\}$; up to replacing $x$ with $-x$, we may assume without loss of generality that
$$\sum_{s\in P}x(s)\geq\frac{1}{2}\sum_{s\in S}|x(s)|.$$
Moreover, our assumption implies the existence of $F\in\F$, with $F\subseteq P$, such that
$$\delta\cdot\sum_{s\in P}x(s)\leq\sum_{s\in F}x(s).$$
Consequently, we obtain
$$\frac{\delta}{2}\cdot\sum_{s\in S}|x(s)|\leq \delta\cdot\sum_{s\in P}x(s)\leq \sum_{s\in F}x(s)\leq \|x\|,$$
which proves our claim. In particular, the completion $X$ of $(c_{00}(S),\n)$ is isomorphic to $\ell_1(S)$.\smallskip

Associated with $F\in2^S$ there is a naturally defined functional $F^*\in X^*$, given by $F^*x:=\sum_{s\in F}x(s)$; note that $F^*$ is well defined for every $F\subseteq S$ in light of the fact that $X$ is isomorphic to $\ell_1(S)$. It is also clear from the definition of $\n$ that $F^*\in B_{X^*}$, whenever $F\in\F$. Moreover, the correspondence $F\mapsto F^*$ defines a function $\Phi\colon\{0,1\}^S\to(X^*,w^*)$, which is easily seen to be continuous and, of course, injective. It readily follows that $\Phi$ establishes an homeomorphism between $\Fc\subseteq\{0,1\}^S$ and $\overline{\F^*}^{w^*}\subseteq B_{X^*}$, where $\F^*:=\Phi(\F)$.\smallskip

Finally, it is a standard fact that $X$ isometrically embeds into $C\left(\overline{\F^*}^{w^*}\right)=C\left(\Fc\right)$, as a consequence of $\F^*$ clearly being $1$-norming for $X$. The fact that $X$ is isomorphic to $\ell_1(S)$ then allows us to conclude the proof. 
\end{proof}

\begin{remark}\label{Rmk: position of l1} Since the argument is completely direct, it is actually possible to keep track of the various embeddings and localise precisely the position of $\ell_1(S)$ into $C\left(\Fc\right)$; more precisely, it is possible to describe the vectors in $C\left(\Fc\right)$ that correspond to the canonical basis of $\ell_1(S)$.

For $s\in S$, let us denote by $\pi_s:\{0,1\}^S\to\{0,1\}$ the canonical projection and let $V_s$ be the clopen set
$$V_s:=\pi_s^{-1}(\{1\})\cap\Fc=\left\{F\in\Fc\colon s\in F\right\}.$$
Inspection of the proof of the previous proposition shows that, assuming $(\dagger)$, the collection $\left(\chi_{V_s}\right)_{s\in S}\subseteq C\left(\Fc\right)$ is equivalent to the canonical basis of $\ell_1(S)$.

A simple modification of an argument given in the course of the proof of Theorem A(ii) will also prove the validity of the converse implication.
\end{remark}

\begin{remark} From the appearance of Rosenthal's celebrated paper \cite{Rosenthal l1}, a well known criterion to prove that a family $(f_\alpha)_{\alpha<\tau}\subseteq B_{C(K)}$ is equivalent to the canonical basis of $\ell_1(\tau)$ consists in showing that, for some reals $r$ and $\delta>0$, the collection of sets
$$\left(\{f_\alpha\leq r\},\{f_\alpha\geq r+\delta\}\right)_{\alpha<\tau}$$
is independent. (We refer to \cite[Proposition 4]{Rosenthal l1} for the definition of the notion of independence and for the simple proof of this claim.) It is perhaps of interest to note that the copy of $\ell_1(S)$ obtained in Proposition \ref{Prop: Ptak gives l1} does not originate from such criterion, unless we are in the trivial case that $\F=[S]^{<\omega}$.

In fact, if there were reals $r$ and $\delta>0$ such that
$$\left(\{\chi_{V_s}\leq r\},\{\chi_{V_s}\geq r+\delta\}\right)_{s\in S}$$ 
is independent, then this would imply that $(V_s^\complement{} ,V_s)_{s\in S}$ is an independent family. As a consequence, for distinct $s_1,\dots,s_n\in S$ we would have
$$\emptyset\neq V_{s_1}\cap\dots\cap V_{s_n}=\left\{F\in\Fc\colon \{s_1,\dots,s_n\} \subseteq F\right\};$$
it would follow from this and $\Fc$ being hereditary that $\{s_1,\dots,s_n\}\in\F$, hence $\F=[S]^{<\omega}$.  
\end{remark}

In conclusion to this section, let us record how the results presented so far imply the validity of the original statement of Pt\'ak's lemma.
\begin{proof}[Proof of Lemma \ref{Ptak II}] We start observing that, without loss of generality, we can assume $|S|=\omega$. Let us, in fact, consider a subset $S_1$ of $S$ such that $|S_1|=\omega$ and the hereditary family $\F\cap S_1:=\{F\in\F\colon F\subseteq S_1\}$. If $\lambda$ is any convex mean on $S_1$, we may extend it to $S$ in the obvious (and unique) way; plainly, if $F\in\F$, $\lambda(F)=\lambda(F\cap S_1)$, where $F\cap S_1\subseteq\F\cap S_1$. Consequently, up to replacing $S$ with $S_1$ and $\F$ with $\F\cap S_1$, we may assume that $|S|=\omega$.\smallskip

Now, Proposition \ref{Prop: Ptak gives l1} yields that $C(\Fc)$ contains a copy of $\ell_1$, which in turn implies that $C(\Fc)$ is not an Asplund space. As a consequence of this, $\Fc$ is necessarily uncountable and it can not be a subset of the countable set $[S]^{<\omega}$. Fact \ref{Fact: from Fclosure to Ptak} leads us to the desired conclusion.
\end{proof}

\section{Pt\'ak's lemma for $\omega_1$}\label{Sec3 Omega1}
This section is dedicated to the proof of Theorem A; both clauses will heavily depend on results from \cite{AMN88}. The proof of claim (i) is essentially the same argument as in the proof of Lemma \ref{Ptak II} given above, but with the Asplund property being replaced by the WLD one. Let us recall a bit of terminology, in order to explain this.\smallskip

A compact space $K$ has property $(M)$ if every regular Borel measure on $K$ has separable support. Here, for the \emph{support} of a Borel measure $\mu$ on $K$, we understand the closed set
$${\rm supp}(\mu):=\{x\in K\colon |\mu|(U)>0\text{ for every neighborhood }U\text{ of }x\}.$$
A topological space $K$ is a \emph{Corson compact} whenever it is homeomorphic to a compact subset $C$ of the product space $[-1,1]^\Gamma$ for some set $\Gamma$, such that every element of $C$ has only countably many non-zero coordinates. A Banach space $X$ is \emph{weakly Lindel\"of determined} (hereinafter, \emph{WLD}) if the dual ball $B_{X^*}$ is a Corson compact in the relative $w^*$-topology.

We shall need the following topological characterisation of WLD Banach spaces of continuous functions, due to Argyros, Mercourakis, and Negrepontis \cite[Theorem 3.5]{AMN88} (we also refer to the same paper, or to \cite{HMVZ, KKLP, Kalenda survey}, for more on Corson compact and WLD spaces).
\begin{theorem}[Argyros, Mercourakis, and Negrepontis, \cite{AMN88}] Let $K$ be a compact topological space. Then $C(K)$ is WLD if and only if $K$ is a Corson compact with property (M).
\end{theorem}

The r\^ole of Martin's axiom MA$_{\omega_1}$ in connection with the above result is that, under MA$_{\omega_1}$, every Corson compact has property (M) (\emph{cf.} \cite[Remark 3.2.3]{AMN88} or \cite[Theorem 5.62]{HMVZ}). More precisely, recall that a compact space $K$ satisfies the \emph{countable chain condition} (\emph{ccc}, for short) if every collection of non-empty disjoint open sets in $K$ is at most countable. It is clear that if $\mu$ is a regular Borel measure on $K$, then the support ${\rm supp}(\mu)$ of $\mu$ is ccc. The previous claim then follows from the fact that, under MA$_{\omega_1}$, every ccc Corson compact is separable (\cite[p. 201, Theorem (b)]{ComNegr chain}, or \cite[p. 207, Exercise (i)]{Fremlin MA}).\smallskip

The combination of the above considerations assures us that, assuming MA$_{\omega_1}$, a compact space $K$ is Corson if and only if $C(K)$ is WLD; by means of this equivalence, we may now readily prove the first part of Theorem A.

\begin{proof}[Proof of Theorem A(i)] According to Proposition \ref{Prop: Ptak gives l1}, $C(\Fc)$ contains an isomorphic copy of $\ell_1(\omega_1)$ and, therefore, it fails to be WLD. Consequently, $\Fc$ is not Corson and it follows immediately that there exists $M\in\Fc$ with $|M|\geq\omega_1$; in fact, if this were false, then the inclusion map $\Fc\subseteq[0,1]^{\omega_1}$ would witness the fact that $\Fc$ is Corson. We may therefore apply Fact \ref{Fact: from Fclosure to Ptak} and conclude the proof.
\end{proof}

As it turns out, assuming some additional set-theoretic axioms is necessary for the validity of the results described above. In particular, the Continuum Hypothesis allows for the construction of Corson compact failing property (M). The first such example was constructed by Kunen in \cite{Kunen corson CH} and one its generalisation, the Kunen--Haydon--Talagrand example, is described in \cite[\S 5]{Negr84}, combining Kunen's construction with Haydon's and Talagrand's examples, \cite{Haydon counterexample, Talagrand corson CH}. Such compact $K$ also has the property that $C(K)$ fails to contain an isomorphic copy of $\ell_1(\omega_1)$. One simpler example, still under CH, is based upon the Erd\H{o}s space and may be found in \cite[Theorem 3.12]{AMN88} or \cite[Theorem 5.60]{HMVZ}. Interestingly, if the Corson compact $K$ is an adequate compact, then $K$ fails to have property (M) if and only if $C(K)$ contains an isomorphic copy of $\ell_1(\omega_1)$ \cite[Theorem 3.13]{AMN88}; in other words, $C(K)$ contains $\ell_1(\omega_1)$, whenever it fails to be WLD.\smallskip

The proof of claim (ii) in Theorem A that we shall give presently will also implicitly depend on the Erd\H{o}s space and is based on a combination of arguments from \cite[Theorems 3.12, 3.13]{AMN88}.
\begin{proof}[Proof of Theorem A(ii)] The assumption of the validity of the Continuum Hypothesis allows us to enumerate in an $\omega_1$-sequence $(K_\alpha)_{\alpha<\omega_1}$ the collection of all compact subsets of $[0,1]$ with positive Lebesgue measure (which, in what follows, we shall denote $\mathscr{L}$). We may also let $(x_\alpha)_{\alpha<\omega_1}$ be a well ordering of the interval $[0,1]$. Since the set $K_\alpha\cap\{x_\beta\}_{\alpha\leq\beta<\omega_1}$ has positive measure, the regularity of $\mathscr{L}$ allows us to select a compact subset $C_\alpha$ of $K_\alpha\cap\{x_\beta\}_{\alpha\leq\beta<\omega_1}$ such that $\mathscr{L}(C_\alpha)>0$, for each $\alpha<\omega_1$. Note that if $A\subseteq\omega_1$ is any uncountable set, then $\sup A=\omega_1$ and it follows that 
$$\bigcap_{\alpha\in A}C_\alpha\subseteq \bigcap_{\alpha\in A}\{x_\beta\}_{\alpha\leq \beta<\omega_1}=\emptyset.$$

We are now in position to define a Corson compact that fails property (M). Consider the set 
$$\mathcal{A}:=\left\{A\subseteq\omega_1\colon \bigcap_{\alpha\in A}C_\alpha\neq \emptyset\right\};$$
if every finite subset of a given set $A$ belongs to $\mathcal{A}$, then the collection of closed sets $\{C_\alpha\}_{\alpha\in A}$ has the finite intersection property and $A\in\mathcal{A}$ follows by compactness. Consequently, $\mathcal{A}$ is an adequate compact. Moreover, the previous consideration shows that every $A\in\mathcal{A}$ is a countable subset of $\omega_1$, whence $\mathcal{A}$ is a Corson compact. The proof that $\mathcal{A}$ fails to have property (M) may be found in \cite[Theorem 3.12]{AMN88}, or \cite[Theorem 5.60]{HMVZ} and we shall not reproduce it here. Therefore, we may fix a positive regular Borel measure $\mu$ on $\mathcal{A}$, whose support is not separable.\smallskip

We now consider again the clopen subsets of $\mathcal{A}$ (\emph{cf.} Remark \ref{Rmk: position of l1})
$$V_\alpha:=\pi_\alpha^{-1}(\{1\})\cap\mathcal{A}=\{A\in\mathcal{A}\colon\alpha\in \mathcal{A}\}\qquad(\alpha<\omega_1)$$
and we shall consider the set $I:=\{\alpha<\omega_1\colon\mu(V_\alpha)>0\}$. Plainly, for $A\in{\rm supp}(\mu)$, we have $\mu(V_\alpha)>0$ whenever $\alpha\in A$; consequently, $A\subseteq I$ and we obtain that ${\rm supp}(\mu)\subseteq2^I$. In light of the fact that the support of $\mu$ is not separable, it follows that $I$ is uncountable. (Here, we are using the fact that every subspace of a separable Corson compact is separable, immediate consequence of the easy observation that separable Corson compact are indeed metrisable.) In turn, we also obtain the existence of an uncountable subset $S$ of $I$ and a real $\delta>0$ such that $\mu(V_\alpha)>\delta$ for $\alpha\in S$.\smallskip

We may now define the desired hereditary family of finite sets: let us consider $\F_0:=\{F\in\mathcal{A}\colon F\text{ is a finite set}\}$ and set $\F:=\{F\in\F_0 \colon F\subseteq S\}$. Clearly, $\Fc\subseteq\overline{\mathcal{F}_0} =\mathcal{A}$; we infer, in particular, that $\Fc$ contains no uncountable set and the conclusion of Pt\'ak's lemma for the cardinal number $\omega_1$ fails to hold for $\F$.

On the other hand, for every convex mean $\lambda$ on $S$ we have
$$\left\Vert\sum_{s\in S}\lambda(s)\chi_{V_s}\right\Vert_{C(\mathcal{A})}\geq \mu\left(\sum_{s\in S}\lambda(s)\chi_{V_s}\right)=\sum_{s\in S}\lambda(s)\mu(V_s)>\delta\cdot\sum_{s\in S}\lambda(s)=\delta.$$
From this strict inequality and $\overline{\mathcal{F}_0}=\mathcal{A}$, we conclude the existence of $F\in\F_0$ such that 
$$\sum_{s\in S}\lambda(s)\chi_{V_s}(F)>\delta;$$
therefore,
$$\delta<\sum_{s\in S}\lambda(s)\chi_{V_s}(F)=\sum_{s\in F\cap S}\lambda(s) =\lambda(F\cap S)\leq\sup_{G\in\F}\lambda(G)$$
and we see that $\F$ satisfies $(\dagger)$.
\end{proof}

\section{Larger cardinals}\label{Sec4 Larger}
In this section we are going to prove Theorem B; before doing this, it will be convenient to recall some results that we shall make us of in the course of the argument.\smallskip

A topological space is \emph{totally disconnected} if every non-empty connected subset is a singleton. Clearly, topological products and subspaces of totally disconnected spaces are totally disconnected.

For a topological space $(X,\mathcal{T})$ and a point $x\in X$, a \emph{local $\pi$-basis} for $x$ (\emph{cf.} \cite[\S 1.15]{Juhasz}) is a family $\mathcal{B}$ of non-empty open subsets of $X$ such that for every neighborhood $V$ of $x$ there exists $B\in\mathcal{B}$ with $B\subseteq V$ (note that $B$ is not required to contain $x$). Every local basis is a local $\pi$-basis, \emph{a fortiori}. The \emph{pseudo-weight} of $(X,\mathcal{T})$ at $x$ is the minimal cardinality of a local $\pi$-basis for $x$.\smallskip

The first ingredient we need is the following result, due to \v{S}apirovski\v{\i}, \cite{Sapirovskii} (see, \emph{e.g.}, \cite[\S 3.18]{Juhasz} or \cite[Theorem 2.11]{Negr84}).
\begin{theorem}[\v{S}apirovski\v{\i}] Let $K$ be a totally disconnected compact topological space and $\kappa$ be an infinite cardinal number. Then there exists a continuous function from $K$ onto $\{0,1\}^\kappa$ if and only if there exists a non-empty closed subset $F$ of $K$ such that the pseudo-weight of $F$ at $x$ is at least $\kappa$, for every $x\in F$.
\end{theorem}

The second building block for our proof is a characterisation, due to Richard Haydon, of those compact spaces whose associated Banach space of continuous functions contains an isomorphic copy of $\ell_1(\kappa)$, for a certain cardinal number $\kappa$. Let us, briefly review some results in this area. 

Pe\l czy\'nski \cite{Pelczynski} and Hagler \cite{Hagler} proved that a Banach space $X$ contains an isomorphic copy of $\ell_1$ (let us write $\ell_1\embed X$, for short) if and only if $L_1[0,1]\embed X^*$. Pe\l czy\'nski also demonstrated that, for an infinite cardinal $\kappa$, $L_1\{0,1\}^\kappa\embed X^*$ whenever $\ell_1(\kappa)\embed X$ and he conjectured the validity of the converse implication. (Here, by $L_1\{0,1\}^\kappa$ we understand the Lebesgue space corresponding to $\{0,1\}^\kappa$, the Borel $\sigma$-algebra and the Haar measure on the compact group $\{0,1\}^\kappa$.) 

The complete solution to Pe\l czy\'nski's conjecture follows from a combination of results due to Argyros and Haydon: Haydon \cite{Haydon counterexample} proved that the conjecture is false for $\kappa=\omega_1$ and assuming CH. On the other hand, Argyros \cite{Argyros nonseparable} proved the correctness of the conjecture for $\kappa\geq\omega_2$ (in ZFC) and for $\kappa=\omega_1$, assuming MA$_{\omega_1}$. A different proof of Argyros' result can be obtained from \cite{Argyros Bourgain Zachariades}; let us also refer to \cite{Haydon Longhorn, Haydon quotient, Haydon survey, Negr84} for a discussion of these and related results.\smallskip

In a related direction, Talagrand \cite{Talagrand} (also see \cite{Argyros shortproof}, for a simplified proof) proved that, for a cardinal number $\kappa$ with ${\rm cf}(\kappa)\geq\omega_1$, $\ell_1(\kappa)\embed X$ if and only if there exists a continuous function from $(B_{X^*},w^*)$ onto $[0,1]^\kappa$. The result we shall need is a similar statement in the case that $X$ is a $C(K)$ space (\emph{cf}. \cite[Remark 2.5]{Haydon quotient}).
\begin{theorem}[Haydon]\label{Haydon's th: l1(k) iff onto map} Let $\kappa$ be a regular cardinal number such that $\lambda^\omega<\kappa$ whenever $\lambda<\kappa$ and let $K$ be a compact topological space. Then $\ell_1(\kappa)\embed C(K)$ if and only if there exists a continuous function from $K$ onto $[0,1]^\kappa$.
\end{theorem}

Having recorded all the results we shall build on, we can now approach the proof of Theorem B. 
\begin{proof}[Proof of Theorem B] Let $\kappa$ be a regular cardinal number such that $\lambda^\omega<\kappa$ whenever $\lambda<\kappa$, let $S$ be a set with $|S|\geq\kappa$ and $\F\subseteq[S]^{<\omega}$ be an hereditary family such that $(\dagger)$ holds. When we combine Proposition \ref{Prop: Ptak gives l1} with Haydon's result, we obtain the existence of a continuous surjection from $\Fc$ to $[0,1]^\kappa$. Consequently, there exists a closed subset $\mathcal{K}_1$ of $\Fc$ that continuously maps onto $\{0,1\}^\kappa$; note that, being a subspace of $\{0,1\}^\kappa$, $\mathcal{K}_1$ is totally disconnected. In light of \v{S}apirovski\v{\i}'s theorem, we conclude that there exists a closed subspace $\mathcal{K}$ of ($\mathcal{K}_1$, hence of) $\Fc$ such that the pseudo-weight of $\mathcal{K}$ at $x$ is at least $\kappa$, for every $x\in\mathcal{K}$. In particular, if $\mathcal{B}$ is any local basis for the topology of $\mathcal{K}$ at any $x\in\mathcal{K}$, then $|\mathcal{B}|\geq\kappa$.\smallskip

Before we proceed, introducing a bit of notation is in order. If $A\in\mathcal{K}$ and $F$ and $G$ are finite subsets of $A$ and $S\setminus A$ respectively, then we shall denote by
$$\mathcal{U}_A(F,G):=\{B\in\mathcal{K}\colon F\subseteq B\subseteq S\setminus G\},$$
a neighborhood of $A$ in $\mathcal{K}$. A particular case of this piece of notation is that
$$\mathcal{U}_A(F,\emptyset):=\{B\in\mathcal{K}\colon F\subseteq B\}.$$
Plainly, 
$$\left\{\mathcal{U}_A(F,G)\colon F\in[A]^{<\omega}, G\in[S\setminus A]^{<\omega}\right\}$$
is a local basis for the topology of $\mathcal{K}$ at $A$.\smallskip

We now note that $\mathcal{K}$ may also be considered as a partially ordered set, with respect to set inclusion. When such partial order is considered, then every chain in $\mathcal{K}$ admits an upper bound in $\mathcal{K}$; in fact, the union of the chain belongs to $\mathcal{K}$, $\mathcal{K}$ being closed in the pointwise topology. An appeal to Zorn's lemma allows us to deduce that there exist in $\mathcal{K}$ maximal elements with respect to inclusion.
\begin{claim*} If $M\in\mathcal{K}$ is any maximal element, then a local basis for the topology of $\mathcal{K}$ at $M$ is given by
$$\mathcal{B}:=\left\{\mathcal{U}_M(F,\emptyset)\colon F\in[M]^{<\omega}\right\}.$$
\end{claim*}
Since clearly $|\mathcal{B}|\leq|M|$, our previous considerations allow us to conclude that, if $M\in\mathcal{K}$ is any maximal element, then $|M|\geq\kappa$. If we select any such maximal element--whose existence we noted above--then Fact \ref{Fact: from Fclosure to Ptak} assures us that $M$ is the set we were looking for. Therefore, in order to conclude the proof, we only need to establish the claim.

\begin{proof}[Proof of the claim] Assume by contradiction that $\mathcal{B}$ is not a local basis. Then there exist finite sets $\tilde{F}$ and $\tilde{G}$ with $\tilde{F}\subseteq M$ and $\tilde{G}\subseteq S\setminus M$ such that no element of $\mathcal{B}$ is contained in $\mathcal{U}_M(\tilde{F},\tilde{G})$. In particular, for every finite set $I$ with $\tilde{F}\subseteq I\subseteq M$ the clopen set  $\mathcal{U}_M(I,\emptyset)\setminus \mathcal{U}_M(\tilde{F},\tilde{G})$ is non-empty. Since evidently $\mathcal{U}_M(I,\emptyset)\cap \mathcal{U}_M(J,\emptyset)=\mathcal{U}_M(I\cup J,\emptyset)$, we deduce that the family of clopen sets

$$\left\{\mathcal{U}_M(I,\emptyset)\setminus \mathcal{U}_M(\tilde{F},\tilde{G})\colon \tilde{F}\subseteq I\subseteq M \text{ and } I \text{ is finite}\right\}$$
has the finite intersection property. The compactness of $\mathcal{K}$ then yields the existence of an element $\tilde{A}$ that belongs to the intersection of the family; it is then immediate to see that $M\subseteq \tilde{A}$. Moreover, the condition $\tilde{A}\notin \mathcal{U}_M(\tilde{F},\tilde{G})$ is equivalent to the fact that $\tilde{A}\cap\tilde{G}\neq \emptyset$ (since $\tilde{F}\subseteq I\subseteq \tilde{A}$). Consequently, $\tilde{A}$ is a proper extension of $M$ (recall that $M\cap\tilde{G}=\emptyset$), thereby contradicting the maximality of $M$ and thus concluding the proof.

\end{proof}\end{proof}

In conclusion to our note, we shall add a few comments on Haydon's result, Theorem \ref{Haydon's th: l1(k) iff onto map}. When \cite{Haydon quotient} appeared, it was unknown whether the equivalence stated in Theorem \ref{Haydon's th: l1(k) iff onto map} (or, more generally, the equivalence between the assertions in \cite[Remark 2.5]{Haydon quotient}) could possibly hold under more general assumptions on $\kappa$. The sufficient condition holding true for every cardinal $\kappa$, Haydon himself (unpublished) later noted that the necessary condition fails to hold for $\kappa=\omega_1$, under the Continuum Hypothesis. One such example was also obtained by N.~Kalamidas, in his Doctoral dissertation (\emph{cf.} \cite[Example 1.3]{Negr84}).

Incidentally, this is also a consequence of the results presented in our note, since the unique point where the proof of Theorem B depends on some cardinality assumption is the appeal to Theorem \ref{Haydon's th: l1(k) iff onto map}; in particular, Theorem B actually holds true for every cardinal number for which the equivalence in Theorem \ref{Haydon's th: l1(k) iff onto map} holds. Theorem A(ii) then yields the desired counterexample.\smallskip

In accordance with Argyros' results on Pe\l czy\'nski's conjecture that we mentioned above, it is natural to conjecture that Haydon's equivalence may actually be valid for every cardinal number $\kappa\geq\omega_2$. This would, of course, imply the validity of Pt\'ak's lemma for every $\kappa\geq\omega_2$.\smallskip

In case that the conjecture were true, it would also lead to a negative answer to the following question.
\begin{problem} Is the existence of a Corson compact $K$ such that $\ell_1(\omega_2)\embed C(K)$ consistent with ZFC?
\end{problem}
Let us just note that, under CH, such a compact space can not exist, in light of Theorem \ref{Haydon's th: l1(k) iff onto map} and the fact that continuous images of Corson compact are Corson compact (\cite{Gulko: image of Corson, MichaelRudin: Corson}), while $[0,1]^{\omega_2}$ is not Corson. Such a compact space also fails to exist under MA$_{\omega_1}$, according to the results we recorded at the beginning of Section \ref{Sec3 Omega1}.

\medskip{}
\textbf{Acknowledgements.} The authors wish to express their gratitude to the anonymous referee for the careful reading of the manuscript and for suggesting a cleaner proof of the claim in the proof of Theorem B.


\end{document}